\documentclass[11pt]{amsart}

\usepackage{amssymb,amsthm,amsmath,enumerate}
\usepackage[numbers,sort&compress]{natbib}
\usepackage{color}
\usepackage{graphicx}
\usepackage{amssymb,amsthm,amsmath}
\usepackage[numbers,sort&compress]{natbib}
\usepackage{color}
\usepackage{graphicx}
\usepackage{tikz}

\title[Embedded eigenvalues]{Sharp bound for embedded eigenvalues  of Dirac operators with decaying potentials}
\author{Vishwam Khapre}
\address{ Department of Mathematics, Texas A\&M University, College Station, TX 77843-3368, USA} \email{vishwam28@tamu.edu }

\author{Kang Lyu}
\address{ School of Mathematics and Statistics, Nanjing University of Science and Technology, Nanjing, 210094, Jiangsu, People’s Republic of China}
\email{lvkang201905@outlook.com}

\author{Andrew Yu}
\address{ Phillips Academy, 180 Main St, Andover, MA 01810, USA} 
\email{andrewyu45@gmail.com }



\newcommand{\abs}[1]{\left\lvert #1 \right\rvert}

\theoremstyle{plain}
\newtheorem{theorem}{Theorem}[section]

\newtheorem{lemma}[theorem]{Lemma}
\newtheorem{proposition}[theorem]{Proposition}

\newcommand{\R}{\mathbb{R}}
\newcommand{\Z}{\mathbb{Z}}
\theoremstyle{definition}

\begin{document}
	\begin{abstract}
		We study  eigenvalues of the Dirac operator with canonical form
		\begin{equation}
		L_{p,q} \begin{pmatrix}
		u \\ v
		\end{pmatrix}=
		\begin{pmatrix}
		0 & -1 \\ 1 & 0
		\end{pmatrix}\frac{d}{dt}
		\begin{pmatrix}
		u \\ v
		\end{pmatrix}+\begin{pmatrix}
		-p & q \\ q & p
		\end{pmatrix}\begin{pmatrix}
		u \\ v
		\end{pmatrix},\nonumber
		\end{equation}	
		where $ p$ and $q$ are real functions.  Under the assumption that
		\begin{equation}
		\limsup_{x\to \infty}x\sqrt{p^2(x)+q^2(x)}<\infty,\nonumber
		\end{equation}
		the essential spectrum of $L_{p,q}$ is $(-\infty,\infty)$.
		We prove that $L_{p,q}$ has no  eigenvalues  if 
		$$\limsup_{x\to \infty}x\sqrt{p^2(x)+q^2(x)}<\frac{1}{2}.$$ Given any   $A\geq \frac{1}{2}$ and  any $\lambda\in\R$, we  construct functions  $p$ and $q$  such that $\limsup_{x\to \infty}x\sqrt{p^2(x)+q^2(x)}=A$ and $\lambda$ is an eigenvalue of the corresponding Dirac operator $L_{p,q}$. We also construct functions $p$ and $q$ so that the corresponding Dirac operator $L_{p,q}$ has any prescribed set {(finitely or countably many)} of eigenvalues.
	\end{abstract}
\maketitle
	\section{Introduction and main results}
	{The Schr\"odinger operator given by
		\begin{align}\label{e1}
		Hu=-u^{\prime\prime}+Vu
		\end{align}
		and the Dirac operator given by
		\begin{equation}\label{Luv}
		L \begin{pmatrix}
		u \\ v
		\end{pmatrix}=
		\begin{pmatrix}
		0 & -1 \\ 1 & 0
		\end{pmatrix}\frac{d}{dx}
		\begin{pmatrix}
		u \\ v
		\end{pmatrix}+\begin{pmatrix}
		p_{11} & p_{12} \\ p_{21} & p_{22}
		\end{pmatrix}\begin{pmatrix}
		u \\ v
		\end{pmatrix}
		\end{equation}
	are two basic models in mathematics and physics.  
		We are interested in the embedded eigenvalue (eigenvalue embeds into the essential spectrum) problem of Schr\"odinger operators and Dirac operators. For Schr\"odinger operators, the problem is well understood. Kato's classical results \cite{kato1959} show that if $\limsup_{x\to \infty}\abs{xV(x)}=A$, then the Schr\"odinger operator has no eigenvalues larger than $A^2$. Wigner and von Neumann's examples \cite{WignervonNeumantype} imply that there exist potentials with $A=8$, such that $\lambda=1$ is an eigenvalue of the associated Schr\"odinger operator. Finally, (see the survey \cite{sim18} for the history), Atkinson and Everitt \cite{atkinson1978bounds} obtained the sharp bound $\frac{4A^2}{\pi^2}$. They proved that there are no eigenvalues larger than $\frac{4A^2}{\pi^2}$, and for any $0<\lambda<\frac{4A^2}{\pi^2}$, there are potentials with 
			$\limsup_{x\to \infty}\abs{xV(x)}=A$ so that $\lambda$ is an eigenvalue of the associated Schr\"odinger operator.}
	
	{The connection between equations \eqref{e1} and \eqref{Luv}
		is straightforward. For example, by letting $p_{11}=V,$ and $p_{12}=p_{21}=p_{22}=0$, one can directly obtain 
	\begin{align}
		-u^{\prime\prime}+\lambda Vu=\lambda^2u\nonumber
\end{align} by 
\begin{equation} L \begin{pmatrix}
		u \\ v
	\end{pmatrix}
	=\lambda\begin{pmatrix}
		u \\ v
	\end{pmatrix}.\nonumber
\end{equation}} In this article, we study embedded eigenvalue problems of a particular type of Dirac operators {on $L^2[0,\infty)\bigoplus L^2[0,\infty)$}, namely Dirac operators with canonical form, 
	\begin{equation}\label{do}
	L_{p,q} \begin{pmatrix}
	u \\ v
	\end{pmatrix}=
	\begin{pmatrix}
	0 & -1 \\ 1 & 0
	\end{pmatrix}\frac{d}{dx}
	\begin{pmatrix}
	u \\ v
	\end{pmatrix}+\begin{pmatrix}
	-p & q \\ q & p
	\end{pmatrix}\begin{pmatrix}
	u \\ v
	\end{pmatrix},
	\end{equation}	
	where $p\in L^2[0,\infty
	)$ and $q\in L^2[0,\infty)$ are real functions (referred to as potentials).
	The canonical form of Dirac operators plays an important role in
	spectral theory \cite[Theorem 5.1]{Remling2018Canonicalbook}. In the study of  asymptotics of eigenvalues and the inverse problems of Dirac operators, it is crucial to use the canonical form
	\cite[pp. 185-187]{Levitan1991introduction}, \cite{Yang20072mDirac,Zhang2021inverseDirac}.	
We refer readers to~\cite{Gesztesy2017oscillation,Gesztesy2020Afunction,Gesztesy2020threshold,Harris1983moreprecise,Remling2020Essential,Hu2019TraceDirac} for more recent development  about various types of Dirac operators.  


	{ For any $\phi_0\in[0,\pi)$, under the boundary condition 
		\begin{align}\label{boundarycondition}
		u(0)\sin\phi_0-v(0)\cos\phi_0=0,
		\end{align}
		the Dirac operator $L_{p,q}$ defined by \eqref{do} is self-adjoint.
	}
	
	{Denote by $\sigma_{ess}(L_{p,q})$ the essential spectrum of $L_{p,q}$. Recall that $\lambda\in \sigma_{ess}(L_{p,q})$ if and only if there is an orthonormal sequence $\{\varphi_n\}_{n=1}^{\infty}$ such that
		\begin{align}
		||L_{p,q}\varphi_n-\lambda\varphi_n||\to 0, \ n\to\infty.\nonumber
		\end{align}}
	It is well known that 
	\begin{align}
	\sigma_{ess}(L_{0,0})=(-\infty,\infty),\nonumber
	\end{align}
	and $L_{0,0}$ has no eigenvalues.

	By \cite[Theorem 6.4]{Weidmann1971oszillation}, if
	\begin{align}
	\sqrt{p^2(x)+q^2(x)}=o(1),\nonumber
	\end{align}
	as $x\to\infty,$ then 
	\begin{align}
	\sigma_{ess}(L_{p,q})=(-\infty,\infty).\nonumber
	\end{align}

	In the first part of our paper, under the assumption that $p$ and $q$ are Coulomb type potentials {(we ask that potentials have no singularity at $x=0$)},
	we study the question when $L_{p,q}$ has embedded eigenvalues.
	
	{	\begin{theorem}\label{ne}
			If $$\limsup_{x\to \infty}x\sqrt{p(x)^2+q(x)^2}=A<\frac{1}{2},$$
			then under any boundary condition \eqref{boundarycondition}, $L_{p,q}$ has no eigenvalues in $(-\infty,\infty)$.
		\end{theorem}
		
		\begin{theorem}\label{ae}
			For any $\phi_0\in[0,\pi)$, $\lambda \in (-\infty,\infty),$ and  $A\geq\frac{1}{2}$, there exist potentials $p$ and $q$ such that 
			$$\limsup_{x\to \infty}x\sqrt{p(x)^2+q(x)^2}=A,$$
			and the Dirac operator $L_{p,q}$ has an eigenvalue $\lambda$ under the boundary condition \eqref{boundarycondition}.
		\end{theorem}
	}

	We say that the potential is $C^{\infty}$ if $p,q$ are $C^{\infty}.$ 
	In the second part of the paper, we will construct $C^{\infty}$ potentials with which $L_{p,q}$ has many embedded eigenvalues.
	
	\begin{theorem}\label{3}
		Let $S=\{\lambda_j\}_{j=1}^N$ be a set of distinct real numbers. Let $\{\theta_j\}_{j=1}^N\subset [0,\pi)$ be a set of angles. There exist $C^{\infty}$ potentials satisfying
		$$\sqrt{p(x)^2+q(x)^2}= \frac{O(1)}{1+x},$$
		{where $O(1)$ depends on $S$}, such that the associated Dirac operator $L_{p,q}$ has {$L^2[0,\infty)\bigoplus L^2[0,\infty)$} solutions $(u_j,v_j)^T$ satisfying
		\begin{eqnarray}
		L_{p,q}\begin{pmatrix}
		u \\v
		\end{pmatrix}=\lambda_j
		\begin{pmatrix}
		u \\ v
		\end{pmatrix}\nonumber
		\end{eqnarray}
		with the boundary condition
		\begin{eqnarray}
		\frac{u(0)}{v(0)}=\cot \theta_j, \nonumber
		\end{eqnarray}	
		for $j=1,\cdots, N.$
	\end{theorem}
	
	\begin{theorem}\label{4}
		Let $S=\{\lambda_j\}_{j=1}^{\infty}$ be a set of distinct real numbers. Let   $\{\theta_j\}_{j=1}^{\infty}\subset[0,\pi)$ be a set of angles. If $h(x)$ is a positive function with {$\lim_{x\to\infty}h(x)=\infty$}, then there exist $C^{\infty}$ potentials satisfying
		$$\sqrt{p(x)^2+q(x)^2}\leq \frac{h(x)}{1+x},$$
		such that the associated Dirac operator $L_{p,q}$ has {$L^2[0,\infty)\bigoplus L^2[0,\infty)$} solutions $(u_j,v_j)^T$ satisfying
		\begin{eqnarray}
		L_{p,q}\begin{pmatrix}
		u \\ v
		\end{pmatrix}=\lambda_j
		\begin{pmatrix}
		u \\ v
		\end{pmatrix}\nonumber
		\end{eqnarray}
		with the boundary condition
		\begin{eqnarray}
		\frac{u(0)}{v(0)}=\cot \theta_j, \nonumber
		\end{eqnarray}	
		for $j=1,2, \cdots.$
	\end{theorem}


	For Dirac operators with single embedded eigenvalue, Evans and Harris \cite{Evans1981bound} obtained the sharp bound for the separated Dirac equation with the form 
	\begin{align}
	\tilde{L} \begin{pmatrix}
	u \\ v
	\end{pmatrix}=
	\begin{pmatrix}
	0 & -1 \\ 1 & 0
	\end{pmatrix}\frac{d}{dt}
	\begin{pmatrix}
	u \\ v
	\end{pmatrix}+\begin{pmatrix}
	p+1 & q \\ q & p-1
	\end{pmatrix}\begin{pmatrix}
	u \\ v
	\end{pmatrix}=\lambda\begin{pmatrix}
	u \\ v
	\end{pmatrix},\nonumber
	\end{align}
	where their results are under the assumption that $q$ is locally absolutely continuous.
	For more results on embedded single eigenvalue, one can refer to \cite{L17,L18dis}.
	
	For many embedded eigenvalues of Schr\"odinger operators or Dirac operators, 
	Naboko \cite{Na86} constructed smooth potentials such that $L_{0,q}$ has dense (rationally independent) embedded eigenvalues. Naboko's constructions work for Schr\"odinger operators as well. Simon \cite{simondense} constructed potentials such that the associated Schr\"odinger operator has dense embedded eigenvalues. More recently, 
	Jitomirskaya and Liu \cite{JL19} introduced a novel idea to construct embedded eigenvalues for Laplacian on manifolds, which is referred to as piecewise constructions. This approach turns out to be quite robust. 
	Liu and his collaborators developed the approach of piecewise constructions to construct embedded eigenvalues for various models \cite{L18dis,Liu2022resonant,LO17,L19stark}. For more results on embedded eigenvalue problems, one can refer to \cite{Schmidtdensedirac,liumana,liupafa}.

	In this paper, we adapt the approach of piecewise construction to study embedded eigenvalue problems of Dirac operators. The main strategy of proofs for our main theorems follow from that of \cite{JL19,LO17,L19stark}. In the current case of Dirac operators, 
	new difficulties and challenges arise from 
	the Dirac operator being vector valued and its potential consisting of a pair of functions $p$ and $q$ (unlike the models in \cite{JL19,L18dis,Liu2022resonant,LO17,L19stark}).

	\section{Proof of Theorems \ref{ne} and \ref{ae}}
	
	
	Let $(u(x), v(x))^T$ be a solution of 
	\begin{equation} L_{p,q} \begin{pmatrix}
	u \\ v
	\end{pmatrix}
	=\lambda\begin{pmatrix}
	u \\ v
	\end{pmatrix}.\nonumber
	\end{equation}
	We define the Pr\"ufer variables  $R(x)$ and $\theta(x)$ of  $\lambda$ by
	\begin{eqnarray}
	u(x)=R(x)\cos \theta(x),\nonumber
	\end{eqnarray}
	and \begin{eqnarray}
	v(x)=R(x)\sin\theta(x).\nonumber
	\end{eqnarray}
	Clearly, we have
	\begin{proposition}\label{prop1}
		Let $R(x)$ and $\theta(x)$ be the Pr\"ufer variables of $\lambda$. Then $\lambda$ is an eigenvalue of the Dirac operator {if and only if} $R\in L^2(0,\infty).$
	\end{proposition}
	By the equation \begin{equation} L_{p,q}  \begin{pmatrix}
	u \\ v
	\end{pmatrix}
	=\lambda\begin{pmatrix}
	u \\ v
	\end{pmatrix},\nonumber
	\end{equation} we obtain
	\begin{eqnarray}\label{r}
	\frac{R^{\prime}}{R}=-q(x)\cos2\theta(x)-p(x)\sin2\theta(x),
	\end{eqnarray}and
	\begin{eqnarray}\label{t}
	\theta^{\prime}=-\lambda+q(x)\sin2\theta(x)-p(x)\cos2\theta(x).
	\end{eqnarray}
	Set $q(x)=V(x)\cos \varphi(x), p(x)=V(x)\sin \varphi (x)$. Note that $p$ and $q$ are completely determined by $V$ and $\varphi$. By \eqref{r} and \eqref{t}, one has
	\begin{eqnarray}\label{crp}
	\frac{R^{\prime}}{R}=-V(x)\cos (2\theta(x)-\varphi(x)),
	\end{eqnarray}
	and 
	\begin{eqnarray}\label{ctp}
	\theta^{\prime}=-\lambda+V(x)\sin (2\theta(x)-\varphi(x)).
	\end{eqnarray}
	{It is obvious that equations \eqref{crp} and \eqref{ctp} are equivalent to \begin{equation} L_{p,q}  \begin{pmatrix}
		u \\ v
		\end{pmatrix}
		=\lambda\begin{pmatrix}
		u \\ v
		\end{pmatrix}.\nonumber
		\end{equation} By Proposition \ref{prop1}, we only need to study \eqref{crp} and \eqref{ctp}.}
	
	\begin{proof}[\textbf{Proof of Theorem \ref{ne}}]Assume 
		\begin{eqnarray}\label{xvx}
		\limsup_{x\to \infty}|xV(x)|=\limsup_{x\to \infty}x\sqrt{p(x)^2+q(x)^2}=A<\frac{1}{2}.
		\end{eqnarray}
		For any $\epsilon>0$ (small enough so that $A+\epsilon<\frac{1}{2}$), there exists $x_0$ so that for any $x>x_0$, one has
		$$ |V(x)|\leq\frac{A+\epsilon}{1+x}.$$
		By \eqref{crp} and \eqref{xvx}, we have
		\begin{align}
		\ln R(x)&=\ln R(x_0)-\int_{x_0}^xV(t)\cos(2\theta(t)-\varphi(t))dt\nonumber\\
		&\geq O(1)-(A+\epsilon)\int_{x_0}^x\frac{1}{1+t}dt\nonumber\\
		&= O(1)-(A+\epsilon)\ln x.\nonumber
		\end{align}
		By the assumption, there exists a positive constant $k$ such that, for large $x$, we have
		$$R(x)\geq kx^{-\frac{1}{2}}.$$
		This implies that $R\notin L^2(0,\infty)$. Hence by Proposition \ref{prop1}, $\lambda$ is not an eigenvalue of $L_{p,q}$. 
	\end{proof}
	
	\begin{proof}[\textbf{Proof of Theorem \ref{ae}} for $A>\frac{1}{2}$]
		We construct $p$ and $q$ as follows:
		\begin{eqnarray}\label{vvv}
		V(x)=\frac{A}{1+x}, \ x\geq0,\nonumber
		\end{eqnarray}
		and 
		\begin{align}
		\varphi(x)=-2\lambda x+2\theta(0),\ x\geq 0.\nonumber
		\end{align}
		By \eqref{ctp} and the uniqueness theorem (see for example \cite[Theorem 2.2]{Teschl2012ODEbook}),  one has for any $x\geq0,$
		$$2\theta(x)-\varphi(x)\equiv 0.$$
		Thus from \eqref{crp} we obtain
		\begin{align}
		\ln R(x)&=\ln R(0)- \int_{0}^x\frac{A}{1+t}dt\nonumber\\
		&=O(1)-A\ln x.\nonumber
		\end{align}
		We immediately obtain that for some small $\epsilon>0$ and any large $x$, 
		$$R(x)\leq x^{-\frac{1}{2}-\epsilon}.$$
		Therefore, $R\in L^2(0,\infty)$ and by Proposition \ref{prop1}, $\lambda$ is an eigenvalue of the corresponding Dirac operator $L_{p,q}$. 
	\end{proof}
	
	\begin{proof}[\textbf{Proof of Theorem \ref{ae}} for $A=\frac{1}{2}$]
		
		Let $\epsilon_n=\frac{1}{2n}$, $a_n=e^{n^3}$. Set
		\begin{align}
		V(x)=\frac{A+\epsilon_n}{x},  \ x\in [a_n, a_{n+1}),\nonumber
		\end{align}
		and 
		\begin{align}
		\varphi(x)=-2\lambda x+2\theta(0).\nonumber
		\end{align}
		By \eqref{ctp} and the uniqueness theorem,  one has for any $x\geq 0$, $$2\theta(x)-\varphi(x)\equiv 0.$$
		
		By \eqref{crp}, one has
		\begin{align}\label{anan}
		\ln R(a_{n+1})-\ln R(a_n)&=-\int_{a_n}^{a_{n+1}}\frac{A+\epsilon_n}{x} dx\nonumber\\
		&=-(A+\epsilon_n)\ln \frac{a_{n+1}}{a_n}.
		\end{align}
		For $t\in [a_n, a_{n+1})$, we have
		\begin{align}\label{tan}
		\ln R(t)-\ln R(a_n)&=-\int_{a_n}^{t}\frac{A+\epsilon_n}{x} dx\nonumber\\
		&=-(A+\epsilon_n)\ln \frac{t}{a_n}.
		\end{align}	
		From \eqref{anan}, we obtain
		\begin{eqnarray}
		\ln R(a_n)=\ln R(a_0)-\sum_{j=0}^{n-1}(A+\epsilon_j)\ln \frac{a_{j+1}}{a_j}\nonumber.
		\end{eqnarray}
		Therefore, one has
		\begin{align}\label{ran}
		R(a_n)&=O(1)e^{-\sum_{j=0}^{n-1}(A+\epsilon_j)\ln \frac{a_{j+1}}{a_j}}\nonumber\\
		&=O(1)\prod_{j=0}^{n-1}a_{j+1}^{-(A+\epsilon_j)}a_j^{A+\epsilon_j}\nonumber\\
		&=O(1)\prod_{j=1}^{n}a_{j}^{-(A+\epsilon_{j-1})}\prod_{j=1}^{n-1}a_j^{A+\epsilon_j}\nonumber\\
		&=O(1)a_{n}^{-(A+\epsilon_{n-1})}\prod_{j=1}^{n-1}a_j^{\epsilon_j-\epsilon_{j-1}}.
		\end{align}
		By \eqref{tan} and \eqref{ran}, we conclude
		\begin{align}\label{Rt}
		R(t)&=O(1) R(a_n)e^{-(A+\epsilon_n)\ln \frac{t}{a_n}}\nonumber\\
		&=O(1) R(a_n)t^{-(A+\epsilon_n)}a_n^{A+\epsilon_n}\nonumber\\
		&=O(1)\prod_{j=1}^{n}a_j^{\epsilon_j-\epsilon_{j-1}}t^{-(A+\epsilon_n)}.
		\end{align}
		It follows that
		\begin{align}
		\int_{a_n}^{a_{n+1}}R(t)^2dt&=O(1)\int_{a_n}^{a_{n+1}}\prod_{j=1}^{n}a_j^{\frac{1}{j}-\frac{1}{j-1}}t^{-1-\frac{1}{n}}dt\nonumber\\
		&\leq O(1)\prod_{j=1}^{n}e^{-j}\frac{n}{e^{n^2}}\nonumber\\
		&\leq O(1)\frac{n}{e^{n^2}}.
		\end{align}
		This implies that $R\in L^2(0,\infty)$,  by Proposition \ref{prop1}, $\lambda$ is an eigenvalue of the corresponding Dirac operator $L_{p,q}$.  
	\end{proof}	
	
	\section{Proof of Theorems \ref{3} and \ref{4}}
	
	
	We assume that $\lambda$ and $\lambda_j$ are different values. Denote the Pr\"ufer variables of $\lambda$ and $\lambda_j$ by $R(x),\theta(x)$ and $R_j(x),\theta_j(x)$, respectively. 
	
	Recall that $V(x)$ and $\varphi(x)$ uniquely determine $p$ and $q$. Define  $V(x)=V(x,b)$ and $\varphi(x)=\varphi(x,\lambda,a, \varphi_0)$ on $[a,\infty)$ by
	\begin{equation}\label{p}
	V(x,b)=\frac{C}{1+x-b},
	\end{equation}
	and
	\begin{align}\label{phix}
	\varphi(x,\lambda, a, \theta_0)=-2\lambda(x-a)+2\varphi_0,
	\end{align}
	where $C$ is a constant will be defined later, $a>b$ and $\varphi_0=\theta(a).$
	
	\begin{lemma}
		Fix $b>0$. Let  $V(x)$ be defined by \eqref{p}. Let $\varphi(x)$ be defined by \eqref{phix}, and $\lambda\neq \lambda_j.$ Let $\theta_j(x)$ be a solution of 
		\begin{equation}\label{tj}
		\theta_j^{\prime}(x)=-\lambda_j+V(x)\sin(2\theta_j(x)-\varphi(x)),
		\end{equation}
		then we have
		\begin{align}\label{ss}
		\int_{x_0}^x\frac{1}{1+t-b}\cos(2\theta_j(t)-\varphi(t))dt=\frac{O(1)}{x_0-b},
		\end{align}
		for any $x>x_0>a$.
	\end{lemma}
	
	\begin{proof}[\textbf{Proof}]
		By \eqref{phix} and \eqref{tj} we have
		\begin{align}
		2\theta_j^{\prime}(t)-\varphi^{\prime}(t)=2(\lambda-\lambda_j)+\frac{O(1)}{1+t-b},\nonumber
		\end{align}
		and 
		\begin{align}
		2\theta_j^{\prime\prime}(t)-\varphi^{\prime\prime}(t)=\frac{O(1)}{1+t-b}.\nonumber
		\end{align}
		It follows that
		\begin{align}
		&\int_{x_0}^x\frac{1}{1+t-b}\cos(2\theta_j(t)-\varphi(t))dt\nonumber\\
		=&\frac{\sin(2\theta_j(t)-\varphi(t))}{2(\lambda-\lambda_j)+\frac{O(1)}{1+t-b}}\frac{1}{1+t-b}\Bigg|^x_{x_0}+O(1)\int_{x_0}^x\frac{1}{(1+t-b)^2}dt\nonumber\\
		=&\frac{O(1)}{x_0-b}.\nonumber
		\end{align}
	\end{proof}

	\begin{lemma}\label{4.2}
		Fix $b>0$. Let $V(x)$ be defined by \eqref{p} on $[a, \infty)$. Let $\varphi(x)$ be defined by \eqref{phix} on $[a, \infty)$, and $\lambda \neq\lambda_j.$ Let $R(x),\theta(x)$ and $ R_j(x),\theta_j(x)$ be the Pr\"ufer variables of $\lambda$ and $\lambda_j$, respectively. For any $x>a$, 
		\begin{eqnarray}\label{lr}
		\ln R(x)-\ln R(a)\leq -100\ln\frac{x-b}{a-b}+C,
		\end{eqnarray}
		\begin{eqnarray}\label{rxa}
		\ln R(x)\leq {\ln R(a)},
		\end{eqnarray}
		{where $C$ is a large constant depending on $\lambda$ and $\lambda_j$,} and for any $ x>x_0 \geq a$ with large enough $x_0-b$, we have
		\begin{eqnarray}\label{rej}
		R_j(x)\leq 1.5R_j(x_0).
		\end{eqnarray}
	\end{lemma}
	
	\begin{proof}[\textbf{Proof}]
		By \eqref{ctp}, \eqref{p} and \eqref{phix}, and the uniqueness theorem, one has
		\begin{align}
		2\theta(x)-\varphi(x)=0.\nonumber
		\end{align}
		Therefore, by \eqref{crp} and \eqref{p}, we have
		\begin{align}
		\ln R(x)&=\ln R(a) - \int_{a}^{x}\frac{C}{1+t-b}dt \nonumber\\
		&=\ln R(a)-C\ln\frac{1+x-b}{1+a-b}.\nonumber
		\end{align}
		Then we immediately obtain \eqref{lr} and \eqref{rxa}.
		
		By \eqref{r} and \eqref{ss}, we have
		\begin{align}
		\ln R_j(x)&=\ln R_j(x_0)-\int_{x_0}^{x}\frac{C\cos(2\theta_j(t)-\varphi(t))}{1+t-b}dt\nonumber\\
		&=\ln R_j(x_0)+\frac{O(1)}{x_0-b}.\nonumber 
		\end{align}
		Hence we obtain \eqref{rej}.
	\end{proof}

	\begin{proposition}\label{ll}
		Let  $\lambda$ and $ S=\{{\lambda}_j\}_{j=1}^k$ be distinct real numbers. 
		Given  $\varphi_0\in[0,\pi)$, if $x_1>x_0>b$, then
		there exist constants $K(\lambda, S)$, $C(\lambda, S)$ (independent of $b, x_0$ and $x_1$) and $ \widetilde V(x,\lambda,S,x_0,x_1,b)\in C^{\infty}$ and $\varphi(x,\lambda,S,x_0,x_1,b,\varphi_0)$
		such that  for $x_0-b>K(\lambda,S)$ the following holds:
		
		\begin{description}
			\item[(1)]   for $x_0\leq x \leq x_1$, ${\rm supp}(\widetilde V)\subset(x_0,x_1)$,  and
			\begin{equation}\label{thm141}
			|\widetilde V(x,\lambda,S,x_0,x_1,b)|\leq \frac{C(\lambda, S)}{x-b}.
			\end{equation}
			
			\item[(2)] the solution of Dirac  equation 
			\begin{equation} L_{p,q} \begin{pmatrix}
			u \\ v
			\end{pmatrix}
			=\lambda\begin{pmatrix}
			u \\ v
			\end{pmatrix},\nonumber
			\end{equation}	
			with the boundary condition $\frac{u(x_0)}{v(x_0)}=\cot\varphi_0$ satisfies
			\begin{equation}\label{thm142}
			R(x_1)\leq C(\lambda, S)\left(\frac{x_1-b}{x_0-b}\right)^{-100} R(x_0),
			\end{equation}
			and  for $x_0<x<x_1$,
			\begin{equation}\label{thm143}
			R(x)\leq  2 R(x_0).
			\end{equation}
			\item[(3)] the solution of Dirac equation
			\begin{equation} L_{p,q} \begin{pmatrix}
			u \\ v
			\end{pmatrix}
			=\lambda_j \begin{pmatrix}
			u \\ v
			\end{pmatrix},\nonumber
			\end{equation}	  with any boundary condition  satisfies for
			$x_0<x\leq x_1$,
			\begin{equation}\label{thm144}
			R_j(x)\leq  2 R_j(x_0).
			\end{equation}
		\end{description}
	\end{proposition}
	\begin{proof}[\textbf{Proof}]
		Let $V(x)$ be given by \eqref{p} and $\varphi(x)$ be given by \eqref{phix}, with $a=x_0$ and $ C=C(\lambda,S)$. Let $x=x_1$ in \eqref{lr}, \eqref{rxa} and \eqref{rej}. We smooth $V(x)$ near $x_0, x_1$ to obtain $\widetilde V(x)$. Notice that by \eqref{crp}, a small perturbation of $V(x)$ will only give a small change of $R(x)$ and $R_j(x)$. Hence Lemma \ref{4.2} still holds with slightly larger constants. We complete the proof.		
	\end{proof}

	\begin{proof}[\textbf{Proof of Theorems \ref{3} and \ref{4}.}]
		
		With the help of Proposition \ref{ll}, the proofs of Theorems \ref{3} and \ref{4} follow from the construction step by step as appearing in \cite{JL19,LO17,L19stark}. {(We mention that although the models in \cite{JL19,LO17,L19stark} are second-order differential equations, the construction still works here because it only relies on Proposition \ref{ll}.)} We only give an outline of the  proof here. Let $\{N_r\}_{r\in\Z^+}$ be a non-decreasing sequence which goes to infinity arbitrarily slowly depending on $h(x)$ \footnote{{For most $r\in\mathbb{N}$, we have $N_{r+1}=N_r$,  and when $N_{r+1}>N_r$, we take $N_{r+1}=N_r+1$. This will  ensure $N_r$ increases to infinity  slowly.}}. We further assume $N_{r+1}=N_r+1$ when $N_{r+1}>N_r$. At the $r$th step, we take $N_r$ eigenvalues into consideration. Applying Proposition 3.3, we construct potentials with $N_r$ pieces, where each piece comes from \eqref{thm141} with $\lambda$ being an eigenvalue.
		The main difficulty is to control the size of each piece (denote by $T_r$). The construction in \cite{JL19,LO17,L19stark} only uses inequalities \eqref{thm141}, \eqref{thm142} and \eqref{thm143} to obtain appropriate $T_r$ and $N_r$. Hence Proposition \ref{ll} implies Theorems \ref{3} and \ref{4}.

	\end{proof}

	\section*{Acknowledgments}
This work was completed as part of the 2022 High School and Undergraduate Research Program  ``STODO" (Spectral Theory Of Differential Operators) at Texas A\&M University.
We would like to thank  Wencai Liu for managing the program, introducing this project and many inspiring discussions. 
 The authors are also grateful to the anonymous referee, whose comments led to an improvement of our manuscript.
	This work was partially supported by NSF DMS-2015683 and DMS-2000345.
	
	\bibliographystyle{abbrv} 
	\bibliography{reference}
	
\end{document}